\documentclass[12pt]{amsart}
\usepackage{latexsym}
\usepackage{amsmath}
\usepackage{amssymb}
\usepackage{amscd}
\usepackage{amsthm}
\usepackage{xypic}
\xyoption{curve}
\usepackage{ifthen} 
\usepackage{hyperref} 
\usepackage{graphicx}
\theoremstyle{plain}
 \newtheorem{theorem}{Theorem}[section]
 \newtheorem{proposition}[theorem]{Proposition}
 \newtheorem{lemma}[theorem]{Lemma}
 \newtheorem{corollary}[theorem]{Corollary}
 
\theoremstyle{definition}
 \newtheorem{definition}[theorem]{Definition}
 \newtheorem{example}[theorem]{Example}
 
 \newtheorem{conjecture}[theorem]{Conjecture}
\theoremstyle{remark}
 \newtheorem{remark}[theorem]{Remark}

\makeindex
\begin{document}
\title[Stable cohomological dimension]{Essential dimension, stable cohomological dimension, and stable cohomology of finite Heisenberg groups}

\author[Bogomolov]{Fedor Bogomolov$^1$}
\address{F. Bogomolov, Courant Institute of Mathematical Sciences\\
251 Mercer St.\\
New York, NY 10012, U.S.A., \emph{and}}
\address{Laboratory of Algebraic Geometry, GU-HSE\\
7 Vavilova Str.\\
Moscow, Russia, 117312}
\email{bogomolo@courant.nyu.edu}
\author[B\"ohning]{Christian B\"ohning$^2$}
\address{Christian B\"ohning, Fachbereich Mathematik der Universit\"at Hamburg\\
Bundesstra\ss e 55\\
20146 Hamburg, Germany}
\email{christian.boehning@math.uni-hamburg.de}

\thanks{$^1$ Supported by NSF grant DMS-1001662 and by AG Laboratory GU- HSE grant RF government ag. 11 11.G34.31.0023}
\thanks{$^2$ Supported by the German Research Foundation (Deutsche Forschungsgemeinschaft) through Heisenberg-Stipendium BO 3699/1-1}

\newcommand{\PP}{\mathbb{P}} 
\newcommand{\QQ}{\mathbb{Q}} 
\newcommand{\ZZ}{\mathbb{Z}} 
\newcommand{\CC}{\mathbb{C}} 
\newcommand{\rmprec}{\wp}
\newcommand{\rmconst}{\mathrm{const}}
\newcommand{\xycenter}[1]{\begin{center}\mbox{\xymatrix{#1}}\end{center}} 
\newboolean{xlabels} 
\newcommand{\xlabel}[1]{ 
                        \label{#1} 
                        \ifthenelse{\boolean{xlabels}} 
                                   {\marginpar[\hfill{\tiny #1}]{{\tiny #1}}} 
                                   {} 
                       } 
\setboolean{xlabels}{false} 

\

\begin{abstract}
We compare the notions of essential dimension and stable cohomological dimension of a finite group $G$, prove that the latter is bounded by the length of any normal series with cyclic quotients for $G$, and show that, however, this bound is not sharp by showing that the stable cohomological dimension of the finite Heisenberg groups $H_p$, $p$ any prime, is equal to two.
\end{abstract}

\maketitle

\section{Introduction}\xlabel{sIntroduction}

Let $G$ be a finite group and $F$ a finite $G$-module. The ground field will be the field of complex numbers unless otherwise specified for all objects below defined over a field. Then we define the \emph{stable cohomology} $H^*_{\mathrm{s}} (G, F)$, following \cite{Bogo93}, compare also the survey article \cite{Bogo05}, as a quotient of the usual group cohomology $H^* (G, F)$ by the ideal $I_{\mathrm{unstable}}$ of unstable classes which is the kernel of the map
\[
H^* (G, F) \to \mathrm{lim}_{U} H^* (U/G, \tilde{F} )
\]
where $U$ runs over (smaller and smaller) $G$-invariant nonempty Zariski open subset of a generically free (complex) linear $G$-representation $V$, contained in the open subset $V^L$ of $V$ on which the $G$-action has trivial stabilizers; $\tilde{F}$ denotes the local system of abelian groups induced by $F$ on (by slight abuse of notation) any of the $U/G$; the kernel $I_{\mathrm{unstable}}$ does not depend on which $V$ we choose.

Alternatively, if $\mathrm{BGal}_V$ denotes the absolute Galois group of $\CC (V)^G$, we can say that $H^*_{\mathrm{s}} (G, F)$ is equal to $H^*_{\mathrm{s}} (G, F)$ modulo the kernel of the map
\[
H^* (\mathrm{B}G, \tilde{F}) \to H^* (\mathrm{BGal}_V , \tilde{F} )
\]
derived from the natural map $f_V: \mathrm{BGal}_V \to \mathrm{B}G$. Both spaces $\mathrm{BGal}_V$ and $\mathrm{B}G$ carry natural topologies, and in general terms one could say that what the present article is about is to get a feeling for the homotopy type, call it $\Omega_G$, of the image of $f_V$ (again $\Omega_G$ is independent of $V$).  Eventually, it would be highly desirable to have an explicit topological construction for $\Omega_G$, starting from group-theoretic or combinatorial data associated to $G$ and its representation theory. This is related to the problem to construct for any given finite group $G$ an explicit discrete group $\Gamma_G$ with a homomorphism $\Gamma_G \to G$  such that $I_{\mathrm{unstable}}$ coincides with the kernel of the map $H^* (G, F) \to H^* (\Gamma_G, F)$. Such a $\Gamma_G$ exists, it suffices to take the fundamental group of a sufficiently small Zariski open $U/G$ in $V^L/G$ in the notation above, where $U/G$ is an Eilenberg-McLane space $K(\Gamma_G, 1)$, but there is no constructive method to obtain $U$ or, phrased differently, to tell how many divisors we have to remove to obtain stabilization.

Hence here we study some first dimension-theoretic properties of $\Omega_G$ in concrete examples. More precisely, the structure and roadmap of the paper is as follows: in Section \ref{sDef} we give a definition for the basic notion of stable cohomological dimension, capturing the dimension theory of $\Omega_G$ from a cohomological point of view. We also discuss the related notion of essential dimension for $G$ which captures the dimension theoretic properties of the class of free $G$-varieties $V^L$ can map equivariantly to. We compare the two notions. Actually, in the case of a finite $p$-group, Merkurjev showed that the essential dimension is equal to the dimension of the smallest faithful representation of that group \cite{KM}, but the stable cohomological dimension is usually much smaller, in fact in Section \ref{sNormalSeries} we prove that the stable cohomological dimension of a finite $p$-group $G$ is bounded by the length of any normal series for $G$ with cyclic quotients. In the subsequent sections we will show that the stable cohomological dimension is frequently even smaller than this bound.

In Section \ref{sHeisenbergHesse} we study stable cohomology of the finite Heisenberg group $H_3$ via a geometric construction taking its starting point from the classical Hesse pencil. In particular, we determine the stable cohomology of $H_3$ with constant coefficients of $H_3$ completely in this geometric way.

Finally, in Section \ref{sHeisenberg} we show that the stable cohomological dimension of the finite Heisenberg groups $H_p$, $p$ any prime, is two, showing that the bound obtained in Section \ref{sNormalSeries} is not sharp. We conclude with a conjecture that the top degree of nonvanishing stable cohomology  with arbitrary coeffcients is the same as that using only constant coefficients.

\section{Definitions and first properties}\xlabel{sDef}

The notion of essential dimension was defined in \cite{Reich}. The latter encapsulates only one but an important
aspect of the behavior of  the so called stable cohomology of groups which was introduced and studied in \cite{Bogo88}, \cite{Bogo93}, \cite{BPT} among others. 

Let $G$ be an algebraic group. The word variety below will not imply that the object is irreducible. If this is intended, we will say irreducible variety.

\begin{definition}\xlabel{dEssDim}
\begin{itemize}
\item[(1)]
A \emph{primitive} $G$-variety is a variety $X$ with an action of $G$ that permutes the components of $X$ transitively. The $G$-action is called \emph{generically free} if there is an open subset $U \subset X$, intersecting each irreducible component nontrivially, on which the $G$-action is free.
\item[(2)]
For a generically free $G$-variety one defines a $G$\emph{-compression} as a dominant $G$-equivariant rational map $X \dasharrow Y$, where $Y$ is another generically free $G$-variety. 
\item[(4)]
The \emph{essential dimension} $\mathrm{ed}(X)$ of a primitive generically free $G$-variety $X$ is the smallest possible value of $\dim (Y/G)$ where $X \dasharrow Y$ is a $G$-compression. The \emph{essential dimension of} $G$, denoted $\mathrm{ed}(G)$, is the essential dimension of a generically free linear $G$-representation.
\end{itemize}
\end{definition}

The essential dimension of $G$ is also equal to the smallest integer $d$ such that any $G$-torsor on a function field over $\CC$ can be induced from a generically free algebraic $G$-scheme of dimension not exceeding $d$, see \cite{GMS}.

Let us introduce the second set of notions around which this paper is centered, going back to \cite{Bogo93}.

\begin{definition}\xlabel{defSCD}
\begin{itemize}
\item[(1)]
Let $X$ be an irreducible variety. The stable cohomological dimension of $X$, denoted $\mathrm{scd} (X)$, is the maximum integer $n$ such that there exists a locally constant sheaf of finite abelian groups $F$ on $X$ with $H^n_{\mathrm{s}} (X, \: F)\neq 0$.
\item[(2)]
Let $G$ be an algebraic group. Then the stable cohomological dimension $\mathrm{scd}(G)$ is defined as the maximum integer $n$ such that there exists a finite $G$-module $M$ with $H^n_{\mathrm{s}} (G, \: M)\neq 0$. 
\end{itemize}
\end{definition} 

\begin{remark}\xlabel{rFinite}
The numbers $\mathrm{scd}(X)$ and $\mathrm{scd} (G)$ are clearly finite: in the first case, it is bounded by $\dim X$, in the second by $\dim V/G$ where $V$ is a generically free $G$-representation. 
\end{remark}

First we state some immediate consequences of the definitions.

\begin{proposition}\xlabel{pStart}
\begin{itemize}
\item[(1)] If $G$ is a finite group and $H\subset G$ a subgroup, then $\mathrm{scd} (G) \ge \mathrm{scd} (H)$. 
\item[(2)] For a finite group $G$ one has \[ \mathrm{scd} (G) = \mathrm{max}_p \{ \mathrm{scd} (\mathrm{Syl}_p (G) ) \} . \]
\item[(3)] 
One always has the inequality $\mathrm{scd} (G) \le \mathrm{ed} (G)$. 
\end{itemize}
\end{proposition}

\begin{proof}
One sees (1) from the equality
\[
H^d (H, \: N) = H^d (G, \: \mathrm{Ind}^G_{H} (N))
\]
for a finite module $N$ over $H$, and the fact that stable classes in the cohomology of the left hand side give stable classes for the cohomology of $G$ on the right hand side (restrict to $H$ to see this).

\

For (2) notice that we have a decomposition of cohomology groups $H^*_{\mathrm{s}} (G, M) = \bigoplus_p H^*_{\mathrm{s}} (G, M_{(p)})$ where $M_{(p)}$ is the $p$-primary component of the finite $G$-module $M$, and an injection \[ H^*_{\mathrm{s}} (G, M_{(p)}) \hookrightarrow \bigoplus_p H^*_{\mathrm{s}} (\mathrm{Syl}_p (G), \: M_{(p)}). \] This shows that $\mathrm{scd} (G) \le \mathrm{max}_p \{ \mathrm{scd} (\mathrm{Syl}_p (G) ) \}$. On the other hand, by (1), $\mathrm{scd} (G) \ge \mathrm{max}_p \{ \mathrm{scd} (\mathrm{Syl}_p (G) ) \}$ holds as well.

\

For (3), let $V$ be a generically free $G$-representation. Let $V \dasharrow Y$ be a $G$-compression inducing a dominant map $V/G \dasharrow Y/G$ with $\dim Y/G \le \mathrm{ed} (G)$. The (topological) map from $V^0/G$ to the classifying space $\mathrm{B}G $ factors:
\[
V^0/G \to Y^0/G \to \mathrm{B}G 
\]
where $V^0$ and $Y^0$ are open subsets where the action of $G$ is free. Hence the stable cohomology of $G$ vanishes in degrees larger than $\dim Y/G$.
\end{proof}

\begin{remark}\xlabel{rBackground}
It was shown in \cite{KM} that the essential dimension of a finite $p$-groups is equal to the minimal dimension of a faithful linear representation of the group. The essential dimension of a finite abelian group equals its rank (see for example \cite{BR}, Thm. 6.1), hence is equal to the stable cohomological dimension. The known values for $\mathrm{ed} (\mathfrak{S}_n)$, the essential dimension of the symmetric group, can be found in \cite{M12}, Thm. 3.22. In particular, $\mathrm{ed} (\mathfrak{S}_7) = 4$ whereas $\mathrm{scd}(\mathfrak{S}_7)=3$; the latter can be seen because the stable cohomological dimension of $\ZZ/2\wr \ZZ/2 \times \ZZ/2 = \mathrm{Syl}_2 (\mathfrak{S}_7)$ is $3$ (cf. \cite{B-B12} where the stable cohomology of wreath products is determined). It shows that one can have strict inequality $\mathrm{scd} (G) < \mathrm{ed} (G)$. 
\end{remark}

\begin{example}\xlabel{eWreath}
For an iterated wreath product $\ZZ/p \wr \ZZ/p \wr \dots \wr \ZZ/p$ ($n$ factors) of cyclic groups of prime order $p$, the essential dimension and stable cohomological dimension coincide and are equal to $p^{n-1}$. This is immediate from \cite{KM} (one has a faithful representation of dimension $p^{n-1}$) and \cite{B-B12}. 
\end{example}

\section{Bounds on stable cohomological dimension via normal series}\xlabel{sNormalSeries}

In this section, we bound $\mathrm{scd}(G)$ for a finite $p$-group $G$ in terms of the length of a normal series for $G$ with cyclic quotients.

\begin{theorem}\xlabel{tNormalSeries}
Let $G$ be a finite $p$-group and let
\[
\{1\} = G_0 \lhd G_1 \lhd \dots \lhd G_n = G
\]
be a normal series for $G$, thus $G_i$ is normal in the next group $G_{i+1}$. Suppose that the quotients $G_{i+1}/G_i$ are cyclic. Then $\mathrm{scd}(G) \le n$.
\end{theorem}

\begin{proof}
We use induction on $n$, the case $n=1$ being clear. Thus we can suppose $H \lhd G$ is a normal subgroup with $G/H \simeq \ZZ/p^k$ cyclic, and we will show that $\mathrm{scd} (G) \le \mathrm{scd} (H) +1$. 

Take a faithful representation $V$ of $G$, and let $V^L$ be an open subvariety of $V$ which is a $\mathrm{K} (\pi , 1)$-space and generically free for the $G$-action. Let $\chi \, :\, G \to \ZZ/p^k$ be a surjection with kernel $H$, and let $\CC_{\chi}$ be a one dimensional representation where $G$ acts via the character $\chi$. The representation $V \oplus \CC_{\chi }$ contains an open subvariety $U:=V^L \times (\CC_{\chi} - \{ 0\} )$ which is a $\mathrm{K} (\pi , 1)$ and on which $G$ acts freely. We obtain a natural homomorphism $\pi_1 (U/G) \to G$ which gives a partial stabilization: the kernel of the map $H^* (G, \: F) \to H^*_{\mathrm{s}} (G, F)$ contains the kernel of the map $H^* (G, F) \to H^* (\pi_1 (U/G), F)$ for an arbitrary finite coefficient module $F$.  Let $G_{\mathbb{Z}}$ be the fiber product
\[
\xymatrix{
G_{\mathbb{Z}} \ar[r]^{\pi_{\mathbb{Z}}} \ar[d] & \mathbb{Z} \ar[d] \\
G \ar[r]^{\chi }  &  \mathbb{Z}/p^k .
}
\]
Then there is a commutative triangle
\[
\xymatrix{
\pi_1 (U/G) \ar[r] \ar[rd] & G_{\mathbb{Z}} \ar[d] \\
  &  G 
}
\]
hence the kernel of $H^* (G, F) \to H^* (G_{\mathbb{Z}}, F)$ is also killed under stabilization. Consider the spectral sequence arising from the projection $G_{\mathbb{Z}} \twoheadrightarrow \mathbb{Z}$ with kernel $H$. Since $H^i (\mathbb{Z}, L) = 0$ for $i> 1$ and any $\mathbb{Z}$-module $L$, the spectral sequence splits into short exact sequences{\small
\[
\xymatrix{ 
0 \ar[r] & H^1 (\mathbb{Z}, H^{i-1} (H, F)) \ar[r] & H^i (G_{\mathbb{Z}}, F) \ar[r] & H^0 (\mathbb{Z}, H^i (H, F)) \ar[r] & 0
}
\]}
and stabilization gives a map between exact sequences{\small 
\[
\xymatrix{
0 \ar[r] & H^1 (\mathbb{Z}, H^{i-1} (H, F)) \ar[r] \ar[d] & H^i (G_{\mathbb{Z}}, F) \ar[r] \ar[d] & H^0 (\mathbb{Z}, H^i (H, F)) \ar[r] \ar[d] & 0\\
0 \ar[r] & H^1 (\mathbb{Z}, H^{i-1} (H_s, F)) \ar[r] & H^i (G_{\mathbb{Z}}, F) \ar[r] & H^0 (\mathbb{Z}, H^i (H_s, F)) \ar[r] & 0
}
\]}
where $H_s$ is a group stabilizing the cohomology of $H$.  Hence, the image of $H^i (G_{\mathbb{Z}}, F)$ is zero in the stable cohomology of the fibration $U \to (\CC_{\chi } - \{ 0 \})$  for $i \ge \mathrm{scd}(H) + 2$. Thus $\mathrm{scd} (G) \le \mathrm{scd}(H) +1$ as claimed. 
\end{proof}

\section{The Heisenberg group $H_3$ and the Hesse pencil} \xlabel{sHeisenbergHesse}

Let $H_3$ be the Heisenberg group over the finite field $\mathbb{F}_3$ sitting in an extension
\[
0 \to \ZZ/3 \to H_3 \to \ZZ/3 \oplus \ZZ/3 \to 0
\]
with the standard generators $x$ and $y$ in the two copies of the quotient $\ZZ/3 \oplus \ZZ/3$ satisfying $x^3 = y^3 = 1$, $x^{-1}y^{-1} xy= z$ where $z$ is a generator of the central $\ZZ/3$. If $E$ is a smooth elliptic curve and $\mathcal{L}$ a degree $3$ line bundle on $E$, then $H_3$ acts on $V^*=H^0 (E, \mathcal{L})$ via the standard Schr\"odinger representation, embedding $E$ into $\PP^2$ as a curve of the Hesse pencil
\[
\lambda (z_1^3 + z_2^3 + z_3^3 ) + \mu z_1z_2z_3 = 0 ,
\]
where the coordinates $z_1, z_2, z_3 \in H^0 (E, \mathcal{L})$ are chosen such that $x\in \ZZ/3$ acts as $z_i \mapsto \omega^i z_i$ with $\omega$ a primitive cube root of unity, and $y$ acts as $z_i \mapsto z_{i-1}$ with lower indices interpreted mod. $3$. For much of the geometry associated to it a good reference is \cite{ArtDolg}.

The nine base points of the pencil, call the set of those $\Sigma$, are the inflection points of any curve in the pencil and located at
\[
(1: -\varrho : 0 ), \quad (0: 1: -\varrho), \quad  (-\varrho: 0:1 ), \quad \; \varrho^3 =1 .
\]
The Hesse pencil gives a rational map $\PP^2 = \PP (V) \dasharrow \PP^1$, $(z_1: z_2: z_3) \mapsto (z_1z_2z_3, z_1^3 +z_2^3 +z_3^3)$, undefined at the nine base points, whence after blowing those up we get a regular map
\[
\varphi : S(3) \to \PP^1
\]
where $S(3)$ is a rational surface having the structure of minimal elliptic surface via $\varphi$ whose fibers are the members of the Hesse pencil. $S(3)$ is a so-called \emph{elliptic modular surface of level} $3$. If $\{ b_1, \dots , b_k \} \subset \PP^1$ are the points corresponding to singular fibers of $S(3)$, $(\PP^1)^0$, $S(3)^0$ the corresponding open complements, $\varphi^0 : S(3)^0 \to (\PP^1)^0$ the induced map on the complement of the singular fibers, then $\varphi^0$ is topologically  a smooth oriented fibration with fiber $S^1 \times S^1$ with a section over $\PP^1$. It is classified by the homomorphism $\pi_1 ((\PP^1)^0 ) \to \mathrm{SL}_2 (\ZZ )$ (and the topology of the base). Its image is the monodromy group of the fibration. In this case it is the modular group
\[
\Gamma (3) = \left\{ \gamma \in \mathrm{SL}_2 (\ZZ ) \mid \gamma \equiv \left( \begin{array}{cc} 1 & 0 \\ 0 & 1 \end{array}\right) (\mathrm{mod}\, 3)\right\} .
\]

\medskip

There is a commutative diagram of fibrations

\begin{center}
\xymatrix{
\CC^{\ast} \simeq \mathrm{B}\ZZ\ar@{^{(}->}[d] \ar[rr]  & & \mathrm{B}(\ZZ/3) \ar@{^{(}->}[d] \\
(V \backslash{\pi^{-1}(\Sigma)})/H_3 \ar[d]^{\pi} \ar[rr] & &  \mathrm{B} H_3 \ar[d]\\
(\PP (V)\backslash \Sigma) / (\ZZ/3)^2 \simeq (S(3) \backslash ( 9 \;\mathrm{sections} ))/(\ZZ/3)^2 \ar[rr] & & \mathrm{B} (\ZZ/3 \oplus \ZZ/3) 
}
\end{center}

and, correspondingly, we have two spectral sequences of fibrations 

\begin{gather*}
E_2^{pq} = H^p ( \ZZ/3\oplus \ZZ/3 , H^q (\ZZ/3 )) \implies H^{p+q} (H_3) , \\
E_2^{pq} = H^p ((S(3) \backslash ( 9 \;\mathrm{sections} ))/(\ZZ/3)^2 , \mathcal{H}^q (\CC^* ) ) \implies H^{p+q} ((V \backslash{\pi^{-1}(\Sigma)})/H_3 ) 
\end{gather*}
(all cohomology taken with some constant coefficients $F$ here, and we write $ \mathcal{H}^q (\CC^* )$ to indicate that this is a local system on the base), 
connected by a morphism of spectral sequences induced by the diagram. We are interested in the image of the cohomology $H^* (H_3)$ in the direct limit of the cohomologies of complements of divisors in $V/H_3$ (i.e., we are interested in $H^*_{\mathrm{s}} (H_3)$). 

\begin{lemma}\xlabel{lImageH2}
For any constant coefficients $F$, the image of the cohomology $H^2 (\ZZ/3 \oplus \ZZ/3, F)$ in the stable cohomology \[ H^2_{\mathrm{s}} (S(3)/ (\ZZ/3 \oplus \ZZ/3), F)\] is zero.
\end{lemma}

\begin{proof}
By e.g. \cite{ArtDolg}, Prop. 5.1, the quotient surface $S(3)/(\ZZ/3 \oplus \ZZ/3)$ has four singular points $s_1, \dots , s_4$ of type $A_2$, given by the orbits of the vertices of the four singular members of the Hesse pencil, which are four triangles. Analytically locally, they are of type $\CC^2/(\ZZ/3)$. Outside of the vertices of the triangles, the action is free. Thus the map 
\[
H^2 (\ZZ/3 \oplus \ZZ/3 , F) \to H^2_{\mathrm{s}} (S(3)/ (\ZZ/3 \oplus \ZZ/3), F)
\]
factors over $S(3)/((\ZZ/3 \oplus \ZZ/3) - \{ s_1, \dots , s_4 \}$ and also, denoting by $\tilde{S}$ the minimal resolution of $S(3)/ (\ZZ/3\oplus \ZZ/3)$, $D^{\pm}_1, \dots , D^{\pm}_4$ the divisors lying over $s_1, \dots , s_4$, the map factors over $H^2 (\tilde{S} - \{ D^{\pm}_1, \dots , D^{\pm}_4 \} , F)$. Now we claim that the image in
\[
H^2 (\mathrm{Gal} (\CC (\tilde{S})), F)
\]
 of every element coming from $H^2 (\ZZ/3 \oplus \ZZ/3 , F)$ has zero residue with respect to any of the divisorial valuations associated to the $D^{\pm}_i$: in fact, the image of the decomposition groups associated to these valuations in $\ZZ/3 \oplus \ZZ/3$ are of type $\ZZ/3$, and the residue maps factor over the stable cohomology of these groups $\ZZ/3$. Hence any element coming from $H^2 (\ZZ/3 \oplus \ZZ/3 , F)$ extends from the open part $\tilde{S} - \{ D^{\pm}_1, \dots , D^{\pm}_4 \}$  through the generic point of any of the divisors $D_i^{\pm}$ and hence to all of $\tilde{S}$. Since $\tilde{S}$ is a rational surface, its second cohomology consists only of classes of algebraic line bundles. All of these are killed under the stabilization morphism. 
\end{proof}

\begin{proposition}\xlabel{pStableH3}
The stable cohomology of the group $H_3$ in degree $3$
\[
H^3_{\mathrm{s}} (H_3, \ZZ/3)
\]
is zero.
The ring $H^*_{\mathrm{s}} (H_3, \ZZ/3)$ can be described by saying that it is generated by elements $y, y'$ of degree $1$, $Y, Y'$ of degree $2$, and all products between any two of these are zero.
\end{proposition}

\begin{proof}
We restrict the left hand fibration in the diagram
\begin{center}
\xymatrix{
\CC^{\ast} \simeq \mathrm{B}\ZZ\ar@{^{(}->}[d] \ar[rr]  & & \mathrm{B}(\ZZ/3) \ar@{^{(}->}[d] \\
(V \backslash{\pi^{-1}(\Sigma)})/H_3 \ar[d]^{\pi} \ar[rr] & &  \mathrm{B} H_3 \ar[d]\\
(\PP (V)\backslash \Sigma) / (\ZZ/3)^2 \simeq (S(3) \backslash ( 9 \;\mathrm{sections} ))/(\ZZ/3)^2 \ar[rr] & & \mathrm{B} (\ZZ/3 \oplus \ZZ/3) 
}
\end{center}
to an open subset $U \subset (\PP (V)\backslash \Sigma) / (\ZZ/3)^2$ chosen so small that 
\begin{itemize}
\item[(a)]
The cohomology of $\tilde{S}$ is stabilized on it.
\item[(b)]
It is a $K(\Gamma , 1)$. 
\end{itemize}
Then we look at the induced map of spectral sequences 
\begin{gather*}
E_2^{pq} = H^p ( \ZZ/3\oplus \ZZ/3 , H^q (\ZZ/3, \ZZ/3 )) \implies H^{p+q} (H_3, \ZZ/3) , \\
E_2^{pq} = H^p (\Gamma , H^q (\ZZ , \ZZ/3) ) \implies H^{p+q} (\pi^{-1}(U), \ZZ/3 ) 
\end{gather*}
to find that the image of $H^3 (H_3, \ZZ/3)$ in $H^3 (\pi^{-1}(U), \ZZ/3)$ is zero; note that any element coming from $H^2 (\ZZ/3\oplus \ZZ/3, \ZZ/3)$ extends to $\tilde{S}$, i.e. the map to the cohomology of $U$ factors through the cohomology of $\tilde{S}$. 

The remaining assertions follow from Theorem 7 of \cite{Leary} and the fact that the Steenrod operations are trivial in stable cohomology. 
\end{proof}

\begin{remark}\xlabel{rTezukaYagita}
This result agrees with what is stated in \cite{TezYag}, Prop. 10.2, though the argument here is different and maybe more elementary.
\end{remark}

\section{Stable cohomological dimension of arbitrary Heisenberg groups}\xlabel{sHeisenberg}

Let $H_p$ be the Heisenberg group $H_p$, $p$ some prime number, which is a central extension

\[
0 \to \ZZ/p \to H_p \to \ZZ/p \oplus \ZZ/p \to 0
\]

defined by thinking of elements of $H_p$ as given by matrices
\[
M(a,b,c)=\left( \begin{array}{ccc} 1 & a & c \\ 0 & 1 & b \\ 0 & 0 & 1 \end{array}\right)
\]
with $a,b,c \in \ZZ/p$ and group composition given by matrix multiplication. Let us write temporarily $A=\ZZ/p$, $\hat{A} = \mathrm{Hom}^{\mathrm{gp}} (\ZZ/p, \CC^*) \simeq \ZZ/p$, and the isomorphism $A \simeq \hat{A}$ is given by the assignment

\[
A \ni k \mapsto  \left( \ZZ/p \ni j \mapsto \chi_k (j) = e^{\frac{2\pi i}{p} k j } \in S^1 \subset \CC^* \right) \in \hat{A} .
\]

Let $V$ be the $p$-dimensional complex vector space of complex valued functions on the set $A$. For a pair $(a , \chi ) \in A \times \hat{A}$ one can define the Weyl operator $W_{(a, \chi )}$, a linear self-map from $V$ to $V$, by

\[
(W_{(a, \chi)} f) (u) = \chi (u) f (a + u), \quad f\in V, u \in A 
\]
and the standard $p$-dimensional Schr\"odinger representation $\varrho$ of $H_p$ on $V$ by 

\[
\varrho (M(a,b,c)) = e^{\frac{2\pi i}{p} c}\, W_{(a, \chi_b)} .
\]

Consider the action of $H_p$ on $V \times T$ where $T = \CC^* \times \CC^*$ and $H_p$ acts on $T$ via the quotient $\Gamma := \ZZ/p \times \ZZ/p$, multiplying each coordinate by a $p$-th root of unity. There is an induced action of the quotient $\Gamma$ on $\PP (V) \times T$ and a $\CC^*$-fibration
\begin{eqnarray}\label{fibration}
f : (V- \{ 0 \}  \times T)/H_p \to (\PP (V) \times T)/\Gamma .
\end{eqnarray}
Let $H \subset \PP (V)$ be a general hyperplane and let $\mathcal{H}$ be its $\Gamma$-orbit. Then $\mathcal{H}$ is a generic hyperplane arrangement and $\Omega := \PP(V) - \mathcal{H}$ has fundamental group $\ZZ^{p^2 -1}$. Moreover, 
\[
(\Omega \times T)/\Gamma \subset (\PP (V) \times T)/\Gamma
\]
has a fundamental group $\Pi$ which is an extension (as is seen by looking at the projection $(\Omega \times T)/\Gamma \to T/\Gamma$)
\begin{eqnarray}\label{ExtensionPi}
0 \to \ZZ^{p^2 -1} \to \Pi \to \ZZ \oplus \ZZ \to 0
\end{eqnarray}
where $\ZZ^{p^2 -1}$ is a $\Gamma$-module (hence a $\ZZ \oplus \ZZ$-module via the surjection $\ZZ \oplus \ZZ \twoheadrightarrow \ZZ/p \oplus \ZZ/p$ which is the product of the canonical projections in both factors), equal to
\[
\ZZ [\Gamma ] / \ZZ \left( \sum_{\gamma \in \Gamma} \gamma \right)
\]
since the fundamental group of the fiber is generated by loops around the hyperplanes in $\mathcal{H}$ and the sum of these is trivial. 

\begin{lemma}\xlabel{lNotSemidirect}
The extension (\ref{ExtensionPi}) is not a semi-direct product.
\end{lemma}

\begin{proof}
Notice first, to get the right picture of the topological situation, that the class of the projective bundle
\[
\omega : (\PP (V) \times T)/\Gamma \to T/\Gamma
\]
in the topological Brauer group of $T/\Gamma \simeq \CC^* \times \CC^*$ is trivial: indeed, the topological Brauer group of any good topological space $X$, say a finite CW-complex, in our case $\CC^* \times \CC^*$, which is homotopy equivalent to $S^1 \times S^1$, is nothing but $H^3_{\mathrm{tors}} (X , \ZZ )$, so this vanishes identically. Since we can find an $S^1 \times S^1$ in the complement of any divisor in $T/\Gamma$, the preceding projective bundle has a section $\Sigma$. However, we claim that there cannot be any such section $\Sigma$ contained in the complement of the hyperplane arrangement $(\Omega \times T)/\Gamma$: indeed, the Picard group of $\PP(V) - \mathcal{H}$ is torsion, hence the $\CC^*$-fibration over $(\Omega \times T)/\Gamma$ corresponds to a torsion line bundle and would restrict trivially to $\Sigma$. Hence the fundamental group $\ZZ^3$ of this restricted $\CC^*$-fibration cannot surject onto the nonabelian group $H_p$.

Note that if we do not remove $\mathcal{H}$, then we can find a $\Sigma$ with a nontrivial $\CC^*$-fibration over it and fundamental group a nontrivial extension of $\ZZ\oplus \ZZ$ by $\ZZ$, surjecting onto $H_p$.
\end{proof}

We need some auxiliary results about stable cohomology of $\Gamma$ with arbitrary coefficients to proceed. 

\begin{proposition}\xlabel{pReduction}
Let $M$ be a finite $\Gamma =\ZZ/p\oplus \ZZ/p$-module annihilated by a power of $p$. Let $c: = \sum_{\gamma \in \Gamma} \gamma \in \ZZ[\Gamma]$ and let $M_c \subset M$ be the submodule consisting of elements annihilated by $c$. Then the natural map
\[
H^2_{\mathrm{s}} (\Gamma , M_c ) \to H^2_{\mathrm{s}} (\Gamma , M)
\]
is surjective.
\end{proposition}

We first remark

\begin{lemma}\xlabel{lReduction}
Let $F$ be a finite $\ZZ/p$ module, $t := \sum_{g\in \ZZ/p} g \in \ZZ[\ZZ/p]$, and $F_t \subset F$ the submodule annihilated by $t$. Then the natural map
\[
H^1 (\ZZ/p, F_t ) \to H^1 (\ZZ/p, F)
\]
is surjective.
\end{lemma}

\begin{proof}
Recall that for any group $G$ and $G$-module $M$, $H^1 (G, M)$ can be interpreted as derivations $f : G \to M$, i.e. maps $f$ from $G$ to $M$ satisfying
\[
f (gh) = f(g) + g f(h) \quad \forall g, h \in G ,
\]
modulo principal derivations, which are those mapping $g\in G$ to $(1-g)m$ for some fixed $m\in M$. Now a derivation $f : \ZZ/p \to F$ is determined by the image of a generator $x$ of $\ZZ/p$, $f(x) =m$ say, but since $x^p = \mathrm{id}$, not every $m\in F$ can occur: since 
\[
0 = f (x^p) = (1 + x + \dots + x^{p-1}) m = t \cdot m
\]
only those $m$ which lie in $F_t$ can be potential images of $x$, but this is the only restriction. Hence 
\[
H^1 (\ZZ/p, F) = F_t / [ \ZZ/p, F] .
\]
The natural map $H^1 (\ZZ/p, F_t) \to H^1 (\ZZ/p, F)$ is nothing other than the map
\[
F_t / [\ZZ/p , F_t] \to F_t / [\ZZ/p, F]
\]
which is visibly surjective.
\end{proof}

Now we turn to the 

\begin{proof}(of Proposition \ref{pReduction})
The stabilization morphism $H^2 (\ZZ/p \oplus \ZZ/p , F ) \to H^2_{\mathrm{s}} (\ZZ/p \oplus \ZZ/p, F)$ factors over the natural map
\[
H^2 (\ZZ/p \oplus \ZZ/p, F ) \to H^2 (\ZZ \oplus \ZZ , F) = H^1 (\ZZ , H^1 (\ZZ, F))
\]
where for the last isomorphism we use the Lyndon-Hochschild-Serre spectral sequence applied to the trivial extension $0 \to \ZZ \to \ZZ \oplus \ZZ \to \ZZ \to 0$. Denote by $C_1, C_2$ the two coordinate copies of $\ZZ/p$ in $\Gamma = \ZZ/p \oplus \ZZ/p$.

{\textbf{Claim}:} The natural map
\[
H^1 (C_1, H^1 (C_2, F_c)) \to H^1 (C_1, H^1 (C_2, F))
\]
is surjective. 

Let us denote by $t_1$ resp. $t_2$ the elements in $\ZZ[ C_i]$, $i=1$ resp. $=2$, which are the sums of all the elements in the group. Then by Lemma \ref{lReduction} we have
that $H^1 (C_1, H^1 (C_2, F))$ is equal to 
\[
H^1 (C_1, F_{t_2}/[C_2, F]) = \{ x \in F_{t_2}/[C_2, F] | t_1 \cdot x = 0 \} / [ C_1, F_{t_2}/[C_2, F] ] .
\]
In other words, elements in this group are given by elements in $F$ which are annihilated by $t_2$ and mapped by multiplication by $t_1$ into $[C_2, F]$, but considered modulo elements in a certain commutator submodule. But every element in $F_{t_2}$ is automatically in $F_c$ since
\[
c = t_1 \cdot t_2 \in \ZZ[\Gamma] , 
\]
hence the Claim follows. 

Now $H^1 (\ZZ/p, H^1 (\ZZ/p, F))$ is in fact a subspace of $H^2 (\ZZ/p \oplus \ZZ/p, F)$ (the assertion of course with $F$ replaced by $F_c$) which surjects onto $H^2_{\mathrm{s}} (\ZZ/p \oplus \ZZ/p, F)$: the fact that it is a subspace follows by looking at the Hochschild-Lyndon-Serre spectral sequence 
\[
E_2^{p,q} = H^p (C_1, H^q (C_2, F)) \implies H^2 (\ZZ/p\oplus \ZZ/p, F)
\]
and remarking that the differential
\[
d_2 : H^1 (C_1, H^1 (C_2, F)) \to H^3 (C_1, H^0 (C_2, F))
\]
is zero because the cohomology of the base $H^3 (C_1, H^0 (C_2, F))$ injects into the cohomology of the total space, since the group extension under consideration is trivial, hence has a section. The terms $H^2 (C_1, H^0 (C_2, F))$ and $H^0 (C_1, H^2 (C_2, F))$ of the $E_2$-page are already mapped to zero under the morphism of spectral sequences to 
\[
E_2^{p,q} = H^p (\ZZ, H^q (\ZZ, F)) \implies H^2 (\ZZ\oplus \ZZ, F)
\]
and, as remarked above, the stabilization morphism factors over $H^2 (\ZZ\oplus \ZZ, F)$. Hence $H^1 (\ZZ/p, H^1 (\ZZ/p, F))$ surjects onto $H^2_{\mathrm{s}} (\ZZ/p \oplus \ZZ/p, F)$.
\end{proof}

\begin{definition}\xlabel{dBasicModules}
Let $M$ be a module for $\Gamma$ (or $\ZZÊ[\Gamma]$, which amounts to the same) which is finite, annihilated by some power of $p$, and annihilated by $c = \sum_{\gamma \in \Gamma} \gamma \in \ZZ [\Gamma ]$. 
We call such an $M$ a \emph{basic module} for $\Gamma$. 
\end{definition}
Hence, if $M$ is a basic $\Gamma$-module, then any map
\[
\ZZ[\Gamma] \to M
\]
onto a one generator submodule factors over $\ZZ[\Gamma]/(c)$. The next Lemma won't be needed for the derivation of the main result, but has some independent interest.

\begin{lemma}\xlabel{lTrivialityOfDiagonal}
If $M$ is a rank one $\ZZ[\Gamma]$-module isomorphic to $(\ZZ/p)^N$ as an abelian group for some $N$ and not free as a $\ZZ/p [\Gamma]$-module, then 
\[
\left( \sum_{\gamma \in \Gamma} \gamma \right) \cdot m
\]
is trivial in $M$, hence $M$ is basic.
\end{lemma}

\begin{proof}
We calculate in $\ZZ/p [\Gamma]$. Write $x := 1 - g_1, y:= 1- g_2 \in \ZZ/p [\Gamma]$, where $g_1$, $g_2$ are generators of the two coordinate copies of $\ZZ/p$ in $\Gamma = \ZZ/p \times \ZZ/p$. Then every $z \in \ZZ/p [\Gamma]$ can be written as
\[
z = \sum_{1 \le i, j \le p-1} a_{ij} x^iy^j, \quad a_{ij} \in \ZZ/p .  
\]
Since $M$ is not free as a $\ZZ/p [\Gamma]$-module, there will be a nonzero such $z$ in the kernel of the surjection $\ZZ/p [\Gamma] \twoheadrightarrow M$ defined by the generator $m$. The kernel being a left-ideal (we consider $M$ as a left $\ZZ/p[\Gamma]$-module), we see that also $x^{p-1}y^{p-1}$ is in the kernel (for this, let $x^{i_0}y^{j_0}$ be the monomial occurring in $z$ with nonzero coefficient which has lexicographically smallest $(i, j)$, and multiply $z$ by $x^{p-1-i_0}y^{p-1-j_0}$, and note that $x^p = y^p =0$).
Now 
\[
x^{p-1}y^{p-1} = \left ( \sum_{i=1}^{p-1} (-1)^i {p-1 \choose i} g_1 \right) \left ( \sum_{i=j}^{p-1} (-1)^j {p-1 \choose j} g_2 \right)
\]
and $(-1)^i {p-1 \choose i} =1$ in $\ZZ/p$, so $x^{p-1}y^{p-1} = \sum_{\gamma \in \Gamma} \gamma$.
\end{proof}

\begin{proposition}\xlabel{pImageH2Basic}
If $M$ is a basic $\Gamma$-module, then the image of $H^2 (\Gamma, M)$ is trivial in $H^2 ( (\Omega \times T)/\Gamma, \tilde{M})$ where $\tilde{M}$ is the local system on $(\Omega \times T)/\Gamma$ induced by $M$.
\end{proposition}

\begin{proof}
We have to show that given a nontrivial extension in $H^2 (\Gamma, M)$:
\[
0 \to M \to H \to \Gamma \to 0
\]
(where by definition $\Gamma$ acts on the abelian normal subgroup $M$ of $H$ via conjugation and this is the given action of $\Gamma$ on $M$), this extension becomes trivial when we pull it back to $\Pi$ via the surjection $\Pi \twoheadrightarrow \ZZ \oplus \ZZ \twoheadrightarrow \ZZ/p \oplus \ZZ/p$:
\[
\xymatrix{
0 \ar[r] & M \ar[r]\ar@{=}[d]  & H_{\Pi} \ar[r]\ar[d] & \Pi \ar[r]\ar[d]\ar@{-->>}[ld] & 0\\
0 \ar[r] & M \ar[r]  & H \ar[r] & \Gamma \ar[r] & 0
}
\]
That the upper extension splits (i.e., is a semi-direct product) is equivalent to the existence of a dashed arrow as shown. In fact, given such an arrow, we get a section of the projection $H_{\Pi} \to \Pi$ since $H_{\Pi}$ is a fiber product of $H$ and $\Pi$ over $\Gamma$; and given a section, we get the required dashed arrow by composing with $H_{\Pi} \to H$.

We claim that we can form some commutative diagram
\[
\xymatrix{
0 \ar[r] & \ZZ^{p^2-1} \ar[r]\ar@{-->>}[d]^{\pi}  & \Pi \ar[r]\ar@{-->>}[d]^{\varphi} & \ZZ \oplus \ZZ \ar[r]\ar@{->>}[d] & 0\\
0 \ar[r] & M \ar[r]  & H \ar[r] & \ZZ/p \oplus \ZZ/p \ar[r] & 0
}
\]
with the right hand vertical arrow $\ZZ \oplus \ZZ \twoheadrightarrow \ZZ/p \oplus \ZZ/p$ the product of the two canonical projections $\ZZ\twoheadrightarrow \ZZ/p$. In other words, we have to justify the existence of the dashed arrows. 

To do this, suppose that $M$ is annihilated by $p^k$; then we know that every submodule of $M$ generated by one element is the image of a map
\[
\ZZ/(p^k)[\Gamma ]/(c) \to M
\]
since $M$ is basic. The reduction map $\ZZ [G]/(c) \twoheadrightarrow \ZZ/(p^k)[G]/(c)$ induces a morphism of extensions

\[
\xymatrix{
0 \ar[r] & \ZZ [\Gamma ]/(c) \ar[r]\ar[d]  & \Pi \ar[r]\ar[d] & \ZZ \oplus \ZZ \ar[r]\ar@{=}[d] & 0\\
0 \ar[r] & \ZZ/(p^k) [\Gamma ]/(c) \ar[r]  & \Pi' \ar[r]\ & \ZZ \oplus \ZZ \ar[r] & 0}
\]

and it suffices to define maps $\varphi'$, $\pi'$, making

\[
\xymatrix{
0 \ar[r] & \ZZ/(p^k) [\Gamma ]/(c) \ar[r]\ar@{-->>}[d]^{\pi'}  & \Pi' \ar[r]\ar@{-->>}[d]^{\varphi'} & \ZZ \oplus \ZZ \ar[r]\ar@{->>}[d] & 0\\
0 \ar[r] & M \ar[r]  & H \ar[r] & \ZZ/p \oplus \ZZ/p \ar[r] & 0
}
\]
commutative.

We want to define the map $\varphi'$ in the following way: let $a= (1,0), b= (0,1) \in \ZZ\oplus \ZZ$, and $\alpha = ( \bar{1}, \bar{0}), \beta = (\bar{0}, \bar{1}) \in \ZZ/p \oplus \ZZ/p = \Gamma$. Then we claim that we get a \emph{well-defined homomorphism} $\varphi'$ by the rule
\[
\tilde{a} \mapsto \tilde{\alpha}, \quad \tilde{b} \mapsto \tilde{\beta }
\]
where $\tilde{a}, \tilde{b}$ are arbitrary lifts of $a,b$ to $\Pi'$ and $\tilde{\alpha},\tilde{\beta }$ equally arbitrary lifts of $\alpha, \beta$ to $H$. In fact, we will show that then $[\tilde{a}, \tilde{b}]$ is a generator of the $\Gamma$-module $\ZZ/(p^k) [\Gamma ]/(c)$ whence the map $\pi'$ is the one mapping this generator to $[\tilde{\alpha}, \tilde{\beta}]$ which generates a rank $1$ submodule of $M$. Then $\pi'$ and with it also $\varphi'$ will be well-defined since  $\ZZ/(p^k) [\Gamma ]/(c)$ can be mapped in a well-defined way into any basic module annihilated by $p^k$ by assigning the image of a generator arbitrarily.

We now use that the extension (\ref{ExtensionPi}) is not a semi-direct product by Lemma \ref{lNotSemidirect} to show that  $\gamma:= [\tilde{a}, \tilde{b}]$ is a generator of the $\Gamma$-module $F:=\ZZ/(p^k) [\Gamma ]/(c)$: in fact, by hypothesis, the class $\overline{\gamma}$ in $\bar{F}:=F/pF$ maps to  a generator of $\bar{F}/[\Gamma , \bar{F}]\simeq \ZZ/p$. Hence, since the action of $\Gamma$ on $\bar{F}$ is nilpotent and this has a filtration by $\Gamma$-invariant subspaces with $\bar{F}/[\Gamma , \bar{F}]\simeq \ZZ/p$ the maximal quotient on which the $\Gamma$-action is trivial, we get that $\overline{\gamma}$ generates $\bar{F}$. But then $\gamma$ must generate $F$ because any proper submodule of $F$ is mapped into a proper submodule of $\bar{F}$ under the reduction map. 
\end{proof}

\begin{corollary}\xlabel{cAnyModule}
If $M$ is any finite $\Gamma$-module, the image of $H^* (\Gamma, M)$ (or what amounts to the same thing, $H^*_{\mathrm{s}} (\Gamma , M)$) in $H^2_{\mathrm{s}} ( (\Omega \times T)/\Gamma, \tilde{M})$ is zero.
\end{corollary}

\begin{proof}
Combine Proposition \ref{pReduction} and Proposition \ref{pImageH2Basic}. 
\end{proof}

\begin{corollary}\xlabel{cStableBasic}
For any finite $H_p$-module $N$, $H^i_{\mathrm{s}} (H_p,  N) =0$ for $i \ge 3$.
\end{corollary}

\begin{proof}
We have three fibrations: first the fibration $F$ given by $\mathrm{B}H_p \to \mathrm{B} \Gamma$ with typical fiber $\mathrm{B}\ZZ/p$, secondly the fibration, call it $F'$, 
\begin{eqnarray}\label{fibration}
f :  f^{-1} ((\Omega \times T)/\Gamma ) \to (\Omega \times T)/\Gamma ,
\end{eqnarray}
and lastly, the fibration $f|_{f^{-1}(U)}$, call it $F''$ where $U \subset (\Omega \times T)/\Gamma$ is a Zariski open subset such that the maps from $H^i (\Gamma , F)$ to $H^i (U, F)$ factors over $H^i_{\mathrm{s}} (\Gamma , F)$ for any $\Gamma$-module $F$ resp. induced local system $F$ on $U$. We get maps of fibrations $F'' \to F' \to F$, and corresponding maps of the associated Serre spectral sequences of the fibrations. 

We want to show that the image of $H^i (\mathrm{B}H_p, \tilde{N})$ in $H^i (f^{-1}(U), \tilde{N})$ is zero for all $i \ge 3$. Now this follows from an inspection of the corrsponding maps of spectral sequences taking into account the following facts:
\begin{itemize}
\item
the image of $H^2 (\mathrm{B}\Gamma , \mathcal{H}^j (\mathrm{B}(\ZZ/p), \tilde{N}))$ in $H^2 ( U , \mathcal{H}^j (\CC^* , \tilde{N}))$ is zero by Corollary \ref{cAnyModule}. 
\item
the image of  $H^i (\mathrm{B}\Gamma , \mathcal{H}^j (\mathrm{B}(\ZZ/p), \tilde{N}))$ in $H^i ( U , \mathcal{H}^j (\CC^* , \tilde{N}))$ for $i\ge 3$ is zero by the choice of $U$;
\item
the image of $H^i (\mathrm{B}\Gamma , \mathcal{H}^j (\mathrm{B}(\ZZ/p), \tilde{N}))$ in $H^i ( U , \mathcal{H}^j (\CC^* , \tilde{N}))$ for $j\ge 2$ is zero since then $H^j (\CC^*, \tilde{N}) = 0$.
\end{itemize}
\end{proof}

\begin{theorem}\xlabel{tHeisenberg}
The stable cohomological dimension of the Heisenberg group $H_p$ is equal to $2$. In particular, the bound given in Theorem \ref{tNormalSeries} (equal to $3$ in the present case) is not sharp. 
\end{theorem}

\begin{proof}
By Corollary \ref{cStableBasic}, the stable cohomological dimension of $H_p$ is $\le 2$ and it suffices to show that 
\[
H^2_{\mathrm{s}} (H_p, \ZZ/p) \neq 0 
\]
(where $\ZZ/p$ are trivial/constant coefficients). We use the computation of the usual group cohomology $H^2 (H_p, \ZZ/p)$ in \cite{Leary}: in particular, it is shown there (Theorem 6 resp. Theorem 7, cf. also the reasoning p. 70, l. 5 ff.) that there exists an element $Y \in H^2 (H_p, \ZZ/p)$ restricting to a product of elements of degree $1$ on a subgroup $\ZZ/p\oplus \ZZ/p$ in $H_p$. Such a $Y$ must be stable. 
\end{proof}

We want to conclude with the following

\begin{conjecture}\xlabel{cSCDPGroups}
Let $G$ be a finite $p$-group. Let $d$ be the stable cohomological dimension of $G$. Then there exists a $G$-module $A$ with trivial action such that $H^d_{\mathrm{s}} (G, A) \neq 0$.
\end{conjecture}

\end{document}